\documentclass{amsart}
\usepackage{amsfonts,amsmath,amssymb,graphicx,
graphics,color,lscape,comment,amsbsy,mathrsfs}

\numberwithin{equation}{section}

\newtheorem{theorem}{Theorem}
\numberwithin{theorem}{section}
\newtheorem{lemma}{Lemma}
\numberwithin{lemma}{section}

\numberwithin{prop}{section}
\newtheorem{corol}{Corollary}
\numberwithin{corol}{section}

\numberwithin{remark}{section}

\numberwithin{defi}{section}

\numberwithin{exe}{section}

\newcommand{\N}{\mathbb{N}}

\newcommand{\notthis}[1]{}

\title[New Inversion Formula]
{\textbf{A New Linear Inversion Formula for a class of Hypergeometric polynomials}}
\author{R. Nasri(*), A. Simonian(*) and F. Guillemin (**)}
\address{Address:  
(*) Orange Labs, OLN/NMP, Orange Gardens, 
44 avenue de la République, CS 50010, 92326 Chatillon Cedex, France
France (**) Orange Labs Networks Lannion, 2 avenue Pierre Marzin, 22307 Lannion Cedex, Lannion, France}
\email{[ridha.nasri, alain.simonian, fabrice.guillemin]@orange.com}

\begin{document}
\date{Version of \today}

\begin{abstract}

Given complex parameters $x$, $\nu$, $\alpha$, $\beta$ and 
$\gamma \notin -\mathbb{N}$, consider the infinite lower triangular matrix 
$\mathbf{A}(x,\nu;\alpha, \beta,\gamma)$ with elements 
$$
A_{n,k}(x,\nu;\alpha,\beta,\gamma) = 
\displaystyle (-1)^k\binom{n+\alpha}{k+\alpha} \cdot    
F(k-n,-(\beta+n)\nu;-(\gamma+n);x)
$$
for $1 \leqslant k \leqslant n$, depending on the Hypergeometric polynomials $F(-n,\cdot;\cdot;x)$, $n \in \mathbb{N}^*$. After stating a general criterion for the inversion of infinite matrices in terms of associated generating functions, we prove that the inverse matrix 
$\mathbf{B}(x,\nu;\alpha, \beta,\gamma) = 
\mathbf{A}(x,\nu;\alpha, \beta,\gamma)^{-1}$ is given by
\begin{align}
B_{n,k}(x,\nu;\alpha, \beta,\gamma) = & \; 
\displaystyle (-1)^k\binom{n+\alpha}{k+\alpha} \; \cdot 
\nonumber \\
& \; \biggl [ \; \frac{\gamma+k}{\beta+k} \, 
F(k-n,(\beta+k)\nu;\gamma+k;x) \; + 
\nonumber \\
& \;  \; \;  
\frac{\beta-\gamma}{\beta+k} \, F(k-n,(\beta+k)\nu;1+\gamma+k;x) \; \biggr ]
\nonumber
\end{align}
for $1 \leqslant k \leqslant n$, thus providing a new class of linear inversion formulas. Functional relations for the generating functions of related sequences $S$ and $T$, that is, 
$T = \mathbf{A}(x,\nu;\alpha, \beta,\gamma) \, S \Longleftrightarrow 
S = \mathbf{B}(x,\nu;\alpha, \beta,\gamma) \, T$, are also provided.

\end{abstract}

\maketitle 


\section{Introduction}


We address a new class of linear inversion formulas with coefficients involving Hypergeometric polynomials. After an overview of the 
state-of-the-art in the associated fields, we then summarize our main contributions.

\subsection{Motivation}
\label{IM}
Consider the following inversion problem:

\textit{\textbf{let $x \in \; ]0,1[$, $\nu < 0$. Solve the infinite lower-triangular linear system}}
\begin{equation}
\forall \, b \in \mathbb{N}^*, \qquad 
\sum_{\ell = 1}^b (-1)^\ell \binom{b}{\ell} \, 
Q_{b,\ell} \, E_{\ell} = K_b,  
\label{T0}
\end{equation}
\textit{\textbf{with unknown $E_\ell$, $\ell \in \mathbb{N}^*$, and where matrix $\mathbf{Q} = (Q_{b,\ell})_{b, \ell \in \mathbb{N}^*}$ is given by}}
\begin{equation}
Q_{b,\ell} = - \frac{\Gamma(b)\Gamma(1-b\nu)}{\Gamma(b-b\nu)} \, 
\frac{x^{1-b}}{1-x} \; F(\ell - b,-b \nu;-b;x), 
\qquad 1 \leqslant \ell \leqslant b.
\label{Q0}
\end{equation}
This inversion problem is motivated by the resolution of an integral equation arising from Queuing Theory [Prop. 5.2]\cite{GQSN18}. In 
(\ref{Q0}), $\Gamma$ is the Euler Gamma function and 
\begin{equation}
F(\alpha,\beta;\gamma;x) = 
\sum_{m \geqslant 0} 
\frac{(\alpha)_m \, (\beta)_m}{(\gamma)_m} \, \frac{x^m}{m!}
\label{DefFpoly}
\end{equation}
denotes the Gauss Hypergeometric series with complex parameters $\alpha$, 
$\beta$, $\gamma \notin -\mathbb{N}$ ($(c)_m$, $m \in \mathbb{N}$, denotes the Pochhammer symbol for any $c \in \mathbb{C}$ with $(c)_0 = 1$ 
\cite[§5.2(iii)]{NIST10}). Recall that function $F(\alpha,\beta;\gamma;\cdot)$ reduces to a polynomial with degree $-\alpha$ (resp. $-\beta$) if $\alpha$ (resp. $\beta$) is a non positive integer; expression (\ref{Q0}) for coefficient $Q_{b,\ell}$ thus involves a Hypergeometric polynomial with degree $b - \ell$ in both $x$ and $\nu$.

The diagonal coefficients $Q_{b,b}$, $b \geqslant 1$, are non-zero so that lower-triangular system (\ref{T0}) has a unique solution. To make this solution explicit in terms of parameters, write system (\ref{T0}) equivalently as 
\begin{equation}
\forall \, b \in \mathbb{N}^*, \qquad 
\sum_{\ell = 1}^b A_{b,\ell}(x,\nu) \, E_\ell = \widetilde{K}_b,
\label{T0BIS}
\end{equation}
with the reduced right-hand side $(\widetilde{K}_b)$ defined by 
$$
\widetilde{K}_b = - \, 
\frac{\Gamma(b-b\nu)}{\Gamma(b)\Gamma(1-b\nu)} (1-x)x^{b-1}\cdot K_b, 
\qquad b \geqslant 1,
$$ 
and with matrix $\mathbf{A}(x,\nu) = (A_{b,\ell}(x,\nu))$ given by
\begin{equation}
A_{b,\ell}(x,\nu) = (-1)^\ell \binom{b}{\ell} F(\ell-b,-b\nu;-b;x), 
\qquad 1 \leqslant \ell \leqslant b.
\label{T0TER}
\end{equation}
As shown in this paper, the linear relation (\ref{T0BIS}) to which initial system (\ref{T0}) has been recast can be explicitly inverted for any right-hand side $(K_b)_{b \in \mathbb{N}^*}$; this consequently fully solves system (\ref{T0}). 

As developed below, our inversion procedure will actually address a larger family $\mathbf{A}(x,\nu;\alpha,\beta,\gamma)$ of infinite matrices depending on three other arbitrary parameters 
$\alpha$, $\beta$, $\gamma$ and including our initial matrix  
(\ref{T0TER}) as a special case. As 
$\mathbf{A}(x,\nu;\alpha,\beta,\gamma)$, the inverse matrix 
$\mathbf{B}(x,\nu;\alpha,\beta,\gamma) = \mathbf{A}(x,\nu;\alpha,\beta,\gamma)^{-1}$ will prove to involve also a specific class of Gauss Hypergeometric polynomials.

\subsection{State-of-the-art}
We first review known classes of linear inversion formulas for the resolution of infinite linear systems. Most of these inversion formulas have been motivated by problems from pure Combinatorics together with the determination of remarkable relations on special functions:

\textbf{a)} given complex sequences $(a_j)_{j \in \mathbb{Z}}$, 
$(b_j)_{j \in \mathbb{Z}}$ and $(c_j)_{j \in \mathbb{Z}}$ with 
$c_j \neq c_k$ for $j \neq k$, it has been shown \cite{Kratten96} that the lower triangular matrices $\mathbf{A}$ and $\mathbf{B}$ with coefficients
\begin{equation}
A_{n,k} = 
\frac{\displaystyle \prod_{j=k}^{n-1}(a_j + b_j c_k)}
{\displaystyle \prod_{j=k+1}^{n}(c_j - c_k)} , 
\qquad
B_{n,k} = \frac{a_k + b_kc_k}{a_n + b_nc_n} \cdot 
\frac{\displaystyle \prod_{j=k+1}^{n}(a_j + b_j c_n)}
{\displaystyle \prod_{j=k}^{n-1}(c_j - c_n)}
\label{Kratt0}
\end{equation}
for $k \leqslant n$, are inverses. A generalization of (\ref{Kratt0}) to the multi-dimensional case when $\mathbf{A} = (A_{\mathbf{n},\mathbf{k}})$ with multi-indexes $\mathbf{n}$, $\mathbf{k} \in \mathbb{Z}^r$, 
$r \in \mathbb{N}$, is also provided in \cite{Schlo97}. As an application, the obtained relations provide  summation formulas for multidimensional basic Hypergeometric series.

The matrix $\mathbf{A} = \mathbf{A}(x,\nu)$ introduced in 
(\ref{T0BIS})-(\ref{T0TER}), however, cannot be cast into the product form (\ref{Kratt0}): in fact, such a product form for the coefficients of 
$\mathbf{A}(x,\nu)$ should involve the $n-k$ zeros 
$c_{j,n,k}$, $k \leqslant j \leqslant n-1$, of the specific Hypergeometric polynomial $F(k-n,-n\nu;-n;x)$, $k \leqslant n$, in variable $x$; but such zeros depend on all indexes $j$, $n$ and $k$, which precludes the use of a factorization such as (\ref{Kratt0}) where sequences with one index only intervene;

\textbf{b)} given a sequence $(\beta_n)_{n \in \mathbb{N}}$, the inverse 
$\mathbf{M}$ of triangular matrix $\mathbf{L}$ with coefficients 
$L_{n,k} = \binom{n}{k} \, \beta_{n-k}$, $k \leqslant n$, has been shown \cite[Theorem 9]{Costa12} to be given by 
$$
M_{n,k} = \binom{n}{k} \, \mathscr{A}_k(0), 
\qquad k \leqslant n,
$$
where $\mathscr{A}_k$, $k \in \mathbb{N}$, is the family of Appell polynomials associated with $(\beta_n)_{n \in \mathbb{N}}$; each coefficient $\mathscr{A}_k(0)$ can be generally written in terms of a determinant of order $k$. This framework does not apply, however, to matrix (\ref{T0TER}) since the ratio $A_{n,k}/\binom{n}{k}$ does not depend there on the difference $n - k$ only;

\textbf{c)} given constants $\alpha$, $\beta$, $x$ and the family of Jacobi polynomials
$$
P_n^{(\alpha,\beta)}(x) = \frac{(\alpha+1)_n}{n!} \, 
F \left (-n,n+\alpha+\beta+1;\alpha+1; \frac{1-x}{2} \right ), 
\quad n \in \mathbb{N},
$$ 
it is established in \cite[Theorem 4.1]{CagKoo15} that the lower triangular matrices $\mathbf{L}$ and $\mathbf{M}$ where 
$L_{n,k}^{(\alpha,\beta)} = P_{n-k}^{(\alpha+k,\beta+k)}(x)$ and 
$$
M_{n,k}^{(\alpha,\beta)} = \displaystyle \frac{n+\beta}{k+\alpha} \, 
P_{n-k}^{(-\alpha-n,-\beta-n)}(x) + 
\frac{\alpha - \beta}{k+\alpha} \, P_{n-k}^{(-\alpha-n,-\beta-n-1)}(x)
$$
for all $k \leqslant n$, are inverses. A $q$-analogue of this result in terms of $q$-ultraspherical polynomials is considered in \cite{Alden15}.\\
Closed analytical formulae for generalized linearization coefficients for Jacobi polynomials and other special polynomials have also been addressed in \cite{Cha10, Fou13}. Writing the latter coefficient $L_{n,k}^{(\alpha,\beta)}$ as
$$
L_{n,k}^{(\alpha,\beta)} = \frac{(\alpha+k+1)_{n-k}}{(n-k)!} \, 
F \left ( k-n,n+k+\alpha+\beta+1;k+\alpha+1;\frac{1-x}{2} \right ),
$$
the second argument in $F$ depends on $n+k$ only. This does not fit, however, the considered case (\ref{T0TER}) where the second argument of 
$F$ is $-n\nu$, with $\nu \neq 1$ in general.

\vspace{0.1in}
In this paper, we will actually consider the inversion of the much larger family of lower triangular matrices 
$\mathbf{A}(x,\nu;\alpha,\beta,\gamma)$ with coefficients 
\begin{equation}
A_{n,k}(x,\nu;\alpha,\beta,\gamma) = 
\displaystyle (-1)^k\binom{n+\alpha}{k+\alpha} \cdot   
F(k-n,-(\beta+n)\nu;-(\gamma+n);x)
\label{T0TER-BIS}
\end{equation}
for $1 \leqslant k \leqslant n$, and depending on three other complex parameters $\alpha$, $\beta$, $\gamma \notin -\mathbb{N}$; our introducing case (\ref{T0TER}) thus corresponds to the specific values 
$\alpha = \beta = \gamma = 0$. Using functional operations on exponential generating series related to the coefficients of matrix 
$\mathbf{A}(x,\nu;\alpha,\beta,\gamma)$, we will show how it can be inverted through a fully explicit procedure. As developed below, the remarkable structure of the inversion 
$\mathbf{B}(x,\nu;\alpha,\beta,\gamma) = 
\mathbf{A}(x,\nu;\alpha,\beta,\gamma)^{-1}$ brings a new contribution to the field of linear inversion formulas, namely infinite matrices with coefficients involving a class of Hypergeometric polynomials depending on five parameters.

\subsection{Paper contribution}
Our main contributions can be summarized as follows:

$\bullet$ in Section \ref{LTS}, we first establish an inversion criterion for a general class of infinite lower-triangular matrices, enabling us to state the inversion formula for the class of lower triangular matrices (\ref{T0TER-BIS});

$\bullet$ in Section \ref{GF}, functional relations are obtained for the ordinary (resp. exponential) generating functions of sequences 
$(S_n)_{n \in \mathbb{N}^*}$ and $(T_n)_{n \in \mathbb{N}^*}$ related by the inversion formula.


\section{Lower-Triangular Systems}
\label{LTS}


Let $(a_m)_{m\in \N}$ and $(b_m)_{m\in \N}$ be complex sequences such that 
$a_0 = b_0 =1$ and denote by $f(x)$ and $g(x)$ their respective exponential generating series, that is,
\begin{equation}
f(x) = \sum_{m=0}^{+\infty}\frac{a_m}{m!} \, x^m, 
\qquad 
g(x) = \sum_{m=0}^{+\infty}\frac{b_m}{m!} \, x^m.
\label{DefFG}
\end{equation}
We use the notation $[x^n] f(x)$ for the coefficient of $x^n$, 
$n \in \mathbb{N}$, in series $f(x)$. For all 
$x, \; \alpha \in \mathbb{C}$, define the infinite lower-triangular matrices $\mathbf{A}(x,\alpha) = (A_{n,k}(x,\alpha))_{n,k \in \N^*}$ and 
$\mathbf{B}(x,\alpha) = (B_{n,k}(x,\alpha))_{n,k \in \N^*}$ by
\begin{equation}
\left\{
\begin{array}{ll}
A_{n,k}(x,\alpha) = \displaystyle 
(-1)^k \binom{\alpha+n}{\alpha+k}\sum_{m=0}^{n-k}\frac{(k-n)_m \, a_m}{m!} \, x^m,
\\ \\
B_{n,k}(x,\alpha) = \displaystyle 
(-1)^k \binom{\alpha+n}{\alpha+k}\sum_{m=0}^{n-k}\frac{(k-n)_m \, b_m}{m!} \, x^m.
\end{array} \right.
\label{DefAB}
\end{equation}
From definition (\ref{DefAB}), matrices $\mathbf{A}(x,\alpha)$ and 
$\mathbf{B}(x,\alpha)$ have diagonal elements 
$$
A_{k,k}(x,\alpha) = B_{k,k}(x,\alpha) = (-1)^k, \qquad 
k \in \mathbb{N}^*,
$$ 
and are thus invertible. 

\subsection{An inversion criterion}
The sequences $(a_m)_{m\in \N}$ and $(b_m)_{m\in \N}$ will be said to be \textit{independent} if, for any pair $(n,k)$, 

$\bullet$ they may depend on one index $n$ or $k$, but not on both,

$\bullet$ they do not depend on the same index, that is, if $(a_m)_{m\in \N}$ depends on $n$ (resp. on $k$), then $(b_m)_{m\in \N}$ depends on $k$ (resp. on $n$) and not on $n$ (resp. not on $k$). 

To alleviate notation, the dependance of either $(a_m)_{m\in \N}$ or 
$(b_m)_{m\in \N}$ with respect to indexes $n$ or $k$ is specified below only when necessary. We now state the following inversion criterion.

\begin{theorem}
\textbf{Given independent sequences $(a_m)_{m\in \N}$ and 
$(b_m)_{m\in \N}$, their associated matrix $\mathbf{A}(x,\alpha)$ and 
$\mathbf{B}(x,\alpha)$ defined by (\ref{DefAB}) are inverse of each other if and only if the condition}
	\begin{equation}
	[x^{n-k}]f(-x)g(x) = \delta_{n,k}, 
	\qquad 1 \leqslant k \leqslant n,
	\label{Invers0}
	\end{equation}
\textbf{on generating functions $f$ and $g$ holds, with $\delta_{n,k} = 1$ if $n = k$ and 0 otherwise.}
\label{theoMaininversionR}
\end{theorem}

\noindent
For conciseness of notation again, the dependence of generating functions 
$f$ and $g$ on parameter $\alpha$ is omitted. The proof of Theorem 
\ref{theoMaininversionR} requires the following technical lemma whose proof is deferred to Appendix \ref{A1}.

\begin{lemma}
\textbf{Let $N \in \N^*$ and complex numbers $\lambda$, $\mu$. Defining}
$$
D_{N}(\lambda,\mu) = 
\sum_{r=0}^{N-1}\frac{(-1)^r}{\Gamma(1+r-\lambda)\Gamma(1-r+\mu)},
$$
\textbf{we then have}
\begin{equation}
D_{N}(\lambda,\mu) = 
\left \{
\begin{array}{ll}
\displaystyle 
\frac{1}{\mu - \lambda} \left[ \frac{1}{\Gamma(-\lambda)\Gamma(1+\mu)} - 
\frac{(-1)^N}{\Gamma(N-\lambda)\Gamma(1-N+\mu)} \right], \, \mu \neq \lambda
\\ \\
\displaystyle \frac{\sin(\pi\lambda)}{\pi} \left[ \psi(-\lambda)-\psi(N-\lambda) \right], ~ \qquad \qquad \qquad \qquad \quad \, \mu = \lambda,
\end{array} \right.
\label{Sn}
\end{equation}
\textbf{where $\psi = \Gamma'/\Gamma$.}
\label{lemm1}
\end{lemma}

\noindent
We now proceed with the justification of Theorem \ref{theoMaininversionR}.

\begin{proof}
Let two independent sequences $(a_m)_{m\in \N}$ and 
$(b_m)_{m\in \N}$. Without loss of generality, we assume that, for every 
$m\in \N$, $a_m$ depends on $n$ and not on $k$, while $b_m$ depends on $k$ and not on $n$.

Both matrices $\mathbf{A}(x,\alpha)$ and $\mathbf{B}(x,\alpha)$ being 
lower-triangular, so is their product $\mathbf{C}(x,\alpha) = 
\mathbf{A}(x,\alpha)\mathbf{B}(x,\alpha)$. After definition (\ref{DefAB}), the coefficient
$$
C_{n,k}(x,\alpha) = \sum_{\ell \geqslant 1} 
A_{n,\ell}(x,\alpha)B_{\ell,k}(x,\alpha), \qquad 
1 \leqslant k \leqslant n
$$
(where the latter sum over index $\ell$ is actually finite), of matrix 
$\mathbf{C}(x,\alpha)$ reads 
\begin{align}
C_{n,k}(x,\alpha) = & \, 
\sum_{\ell = 1}^{+\infty} (-1)^\ell \, \binom{\alpha+n}{\alpha+\ell} 
\sum_{m=0}^{n-\ell} \frac{(-1)^m(n-\ell)! \, a_m}{(n-\ell-m)!m!} \, x^m \; \times  
\nonumber \\
& \, (-1)^k \, \binom{\alpha+\ell}{\alpha+k}\sum_{m'=0}^{\ell-k} 
\frac{(-1)^{m'}(\ell-k)! \, b_{m'}}{(\ell-k-m')!m'!} \, x^{m'}
\nonumber
\end{align}
after writing $(-r)_m = (-1)^m r!/(r-m)!$ for any positive integer $r$.
Using the identity 
$$
\binom{\alpha+n}{\alpha+\ell}\binom{\alpha+\ell}{\alpha+k}
(n-\ell)! (\ell-k)!=\frac{\Gamma(\alpha+n+1)}{\Gamma(\alpha+k+1)},
$$
$C_{n,k}(x,\alpha)$ simplifies to
\begin{align}
C_{n,k}(x,\alpha) = & \, 
(-1)^k \, \frac{\Gamma(\alpha+n+1)}{\Gamma(\alpha+k+1)} \; \times 
\nonumber \\
& \, \sum_{\ell = 1}^{+\infty} (-1)^\ell 
\sum_{m = 0}^{n-\ell} \frac{(-1)^m a_m \, x^m}{m!(n-\ell-m)!} 
\sum_{m' = 0}^{\ell - k} \frac{(-1)^{m'} b_{m'} \, x^{m'}}{m'!(\ell-k-m')!}.
\label{P11}
\end{align}
Since $a_m$ depends only on $n$ and $b_m$ depends only on $k$, $a_m$ and 
$b_m$ are both independent of index $\ell$ in the sum (\ref{P11}). We can then exchange the summation order to sum first on $\ell$, giving
\begin{align}
C_{n,k}(x,\alpha) = 
(-1)^k \, \frac{\Gamma(\alpha+n+1)}{\Gamma(\alpha+k+1)} & \, 
\sum_{(m,m') \in \Delta_{n,k}} \frac{(-1)^m a_m \, x^m}{m!}
\frac{(-1)^{m'} b_{m'} \, x^{m'}}{m'!} \; \times
\nonumber \\
& \; \; \; \; \, \sum_{k \leqslant \ell \leqslant n}
\frac{(-1)^\ell}{(n-\ell-m)!(\ell-k-m')!}
\label{P12}
\end{align}
with the subset 
$\Delta_{n,k} = \{(m,m') \in \mathbb{N}^2, \; m+m' \leqslant n-k\}$ for given $k \leqslant n$, and where the latter sum on index $\ell$ can be equivalently written as
\begin{align}
\sum_{k \leqslant \ell \leqslant n}\frac{(-1)^\ell}{(n-\ell-m)!(\ell-k-m')!} 
= & \, \sum_{r = 0}^{n-k} \frac{(-1)^{n-r}}{(r-m)!(n-r-k-m')!} 
\nonumber \\
= & \, (-1)^n \, D_{n-k+1}(m,n-k-m')
\nonumber
\end{align}
with the index change $\ell = n-r$ and the notation of Lemma \ref{lemm1}. The expression (\ref{P12}) for coefficient $C_{n,k}(x)$ consequently reduces to
\begin{align}
C_{n,k}(x,\alpha) = 
(-1)^{n+k} \, \frac{\Gamma(\alpha+n+1)}{\Gamma(\alpha+k+1)} \, 
\sum_{(m,m') \in \Delta_{n,k}} & \frac{(-1)^m a_m \, x^m}{m!}
\frac{(-1)^{m'} b_{m'} \, x^{m'}}{m'!} \; \times
\nonumber \\
& D_{n-k+1}(m,n-k-m')
\label{P13}
\end{align}
and we are left to calculate $D_{n-k+1}(m,n-k-m')$ for all non negative 
$m$ and $m'$, $(m,m') \in \Delta_{n,k}$. By Lemma \ref{lemm1} applied to 
$\lambda = m$ and $\mu = n-k-m'$, we successively derive that:

\begin{itemize}
\item[\textbf{(a)}] if $\mu > \lambda \Leftrightarrow m + m' < n-k$, formula 
(\ref{Sn}) entails
\begin{align}
& D_{n-k+1}(m,n-k-m') \; = 
\nonumber \\ 
& \frac{1}{n-k-(m+m')} \left [ \frac{1}{\Gamma(-m)\Gamma(1+n-k-m')} - 
\frac{(-1)^{n-k+1}}{\Gamma(n-k+1-m)\Gamma(-m')} \right ];
\nonumber
\end{align}
as $\Gamma(-m) = \Gamma(-m') = \infty$ for all non negative integers 
$m \geqslant 0$ and $m' \geqslant 0$, each fraction of the latter expression vanishes and thus
\begin{equation}
D_{n-k+1}(m,n-k-m') = 0, \qquad m + m' < n-k;
\label{P14}
\end{equation}

\item[\textbf{(b)}] if $\lambda = \mu \Leftrightarrow m + m' = n-k$,  formula (\ref{Sn}) entails
\begin{equation}
D_{n-k+1}(m,m) = \lim_{\lambda \rightarrow m} \frac{\sin(\pi\lambda)}{\pi} 
\left [ \psi(-\lambda) - \psi(n-k+1-\lambda) \right ].
\label{P15}
\end{equation}
We have $\sin(m \pi) = 0$ while function $\psi$ has a polar singularity at every non positive integer; the limit (\ref{P15}) is thus indeterminate 
($0 \times \infty$) but this is solved via the reflection 
formula $\psi(z) - \psi(1-z) = - \pi \, \cot(\pi \, z)$, 
$z \notin - \mathbb{N}$, for function $\psi$ 
\cite[Chap.5, §5.5.4]{NIST10}. In fact, applying the latter to 
$z = -\lambda$ first gives
$\sin(\pi \lambda) \, \psi(-\lambda) = 
\sin(\pi \lambda) \, \psi(1+\lambda) + \pi \cdot \cos(\pi\lambda)$ 
whence
$$
\lim_{\lambda \rightarrow m} 
\frac{\sin(\pi \lambda)}{\pi} \, \psi(-\lambda) 
= 0 \times \psi(1+m) + (-1)^m = (-1)^m;
$$
besides, the second term $\psi(n-k+1-\lambda)$ in (\ref{P15}) has a finite limit when $\lambda \rightarrow m$ since $m + m' = n - k \Rightarrow m \leqslant n - k$ so that $n-k+1-\lambda$ tends to a positive integer. From 
(\ref{P15}) and the latter discussion, we are left with
\begin{equation}
D_{n-k+1}(m,m) = (-1)^m, \qquad m + m' = n-k.
\label{P16}
\end{equation}
\end{itemize}

In view of items \textbf{(a)} and \textbf{(b)}, identities 
(\ref{P15}) and (\ref{P16}) together reduce expression (\ref{P13}) to
\begin{align}
C_{n,k}(x,\alpha) = & \; \frac{\Gamma(\alpha+n+1)}{\Gamma(\alpha+k+1)} \, 
\sum_{m = 0}^{n-k} \frac{(-1)^m a_m \, x^m}{m!} \, 
\frac{b_{n-k-m}}{(n-k-m)!}x^{n-k} 
\nonumber \\
= & \; \frac{\Gamma(\alpha+n+1)}{\Gamma(\alpha+k+1)} \; [x]^{n-k}f(-x)g(x)
\nonumber
\end{align}
where $f$ and $g$ denote the exponential generating function of the sequence $(a_m)_{m \in \mathbb{N}}$ and the sequence 
$(b_m)_{m \in \mathbb{N}}$, respectively. It follows that 
$\mathbf{C}(x,\alpha) = \mathbf{A}(x,\alpha) \mathbf{B}(x,\alpha)$ is the identity matrix \textbf{Id} if and only if condition (\ref{Invers0}) holds, as claimed. 
\end{proof}

\subsection{The inversion formula}
We now formulate the inversion formula for a whole family of lower-triangular matrices involving Hypergeometric polynomials.

\begin{theorem}
\textbf{Let $x, \, \nu, \, \alpha, \, \beta \in \mathbb{C}$ and 
$\gamma \in \mathbb{C} \setminus \, \mathbb{Z^*_{-}}$. Define the lower-triangular matrices $\mathbf{A}(x,\nu;\alpha, \beta,\gamma)$ and 
$\mathbf{B}(x,\nu;\alpha, \beta,\gamma)$ by}
\begin{equation}
\left\{
\begin{array}{ll}
A_{n,k}(x,\nu;\alpha, \beta,\gamma) = 
\displaystyle (-1)^k\binom{n+\alpha}{k+\alpha} \,   
F(k-n,-(\beta+n)\nu;-(\gamma+n);x),
\\ \\
B_{n,k}(x,\nu;\alpha, \beta,\gamma) = 
\displaystyle (-1)^k\binom{n+\alpha}{k+\alpha} \; \cdot 
\\ 
\qquad \qquad \qquad \qquad \; \; 
\displaystyle 
\bigg[ \; \frac{\gamma+k}{\beta+k} \, F(k-n,(\beta+k)\nu;\gamma+k;x) 
\; + 
\\
\qquad \qquad \qquad \qquad \; \; \; \; 
\displaystyle 
\frac{\beta-\gamma}{\beta+k} \, F(k-n,(\beta+k)\nu;1+\gamma+k;x) \; \bigg]
\end{array} \right.
\label{DefABxNU}
\end{equation}
\textbf{for $1 \leqslant k \leqslant n$. The inversion formula}
\begin{equation}
T_n = \sum_{k=1}^{n}A_{n,k}(x,\nu)S_k 
\Longleftrightarrow 
S_n = \sum_{k=1}^{n}B_{n,k}(x,\nu)T_k, \quad n \in \mathbb{N}^*,
\label{eq:inversionR}
\end{equation}
\textbf{then holds for any pair of complex sequences 
$(S_n)_{n \in \mathbb{N}^*}$ and $(T_n)_{n \in \mathbb{N}^*}$.}
\label{PropIn}
\end{theorem}


\begin{proof}
To show that $\mathbf{A} \mathbf{B} = \mathrm{\textbf{Id}}$ for matrices 
$\mathbf{A}$ and $\mathbf{B}$ defined in (\ref{DefABxNU}), we apply  Theorem \ref{theoMaininversionR}. From definition (\ref{DefAB}), we first specify the sequences $(a_{m;n,k})_{m \in \mathbb{N}}$ and 
$(b_{m;n,k})_{m \in \mathbb{N}}$ for a given pair $(n,k)$.

From the standard definition (\ref{DefFpoly}) of Gauss Hypergeometric function $F(\alpha,\beta,\gamma;\cdot)$, the sequences 
$(a_{m;n,k})_{m \in \mathbb{N}}$ and $(b_{m;n,k})_{m \in \mathbb{N}}$ respectively associated with matrices $\mathbf{A}$ and $\mathbf{B}$ defined in (\ref{DefABxNU}) are readily given by
$$
\left\{
\begin{array}{ll}
a_{m;n,k} = \displaystyle 
\frac{(-(\beta+n)\nu)_m}{(-\gamma-n)_m}, 
\\ \\
b_{m;n,k} = \displaystyle 
\frac{\gamma+k}{\beta+k} \, \frac{((\beta+k)\nu)_m}{(\gamma+k)_m} + 
\frac{\beta-\gamma}{\beta+k} \, \frac{((\beta+k)\nu)_m}{(1+\gamma+k)_m}
\end{array} \right.
$$ 
for all $m \geqslant 0$ and given $n, k \in \mathbb{N}^*$, 
$1 \leqslant k \leqslant n$; in particular, $a_{0;n,k}=b_{0;n,k} = 1$. It appears that $a_{m;n,k}$ does not depend on $k$, while $b_{m;n,k}$ does not depend on $n$;  sequences $(a_{m;n,k})$ and $(b_{m;n,k})$ are therefore independent. In the following, we remove the index $k$ from 
$a_{m;n,k}$ and the index $n$ from $b_{m;n,k}$.

To verify the inversion criterion (\ref{Invers0}), let $f_n$ and $g_k$ denote the exponential generating function of sequence 
$(a_{m;n})_{m \geqslant 0}$ and $(b_{m;k})_{m \geqslant 0}$, respectively. The product $f_n(-x)g_k(x)$ is then given by
$$
f_n(-x)g_{k}(x) = \left ( \sum_{m \geqslant 0} (-1)^m 
\frac{a_{m;n}}{m!} \, x^m \right ) 
\left ( \sum_{m \geqslant 0} \frac{b_{m;k}}{m!} \, x^m \right )
= \sum_{\ell \geqslant 0} U_\ell^{(n,k)} \, x^\ell
$$
where we set, for any $\ell \geqslant 0$, 
\begin{align}
U_\ell^{(n,k)} = & \, \sum_{m=0}^\ell (-1)^m \frac{(-(\beta+n)\nu)_m}{(-\gamma-n)_m \, m!} \; \times 
\nonumber \\
& \, \left[\frac{\gamma+k}{\beta+k} \, 
\frac{((\beta+k)\nu)_{\ell-m}}{(\ell-m)!(\gamma+k)_{\ell-m}} + 
\frac{\beta-\gamma}{\beta+k} \, 
\frac{((\beta+k)\nu)_{\ell-m}}{(\ell-m)!(1+\gamma+k)_{\ell-m}}\right].
\label{DefU}
\end{align}
Criterion (\ref{Invers0}) thus amounts to show that $U_{n-k}^{(n,k)} = 1$ if $n = k$ and $U_{n-k}^{(n,k)} = 0$ if $1 \leqslant k < n$. Let then 
$n \geqslant k$ and consider expression (\ref{DefU}) applied to 
$\ell = n-k$; using the identities
\begin{align}
\frac{1}{(-\gamma-n)_m (\gamma+k)_{n-k-m}} = & \, 
(-1)^m \, \frac{(\gamma+n-m)\Gamma(\gamma+k)}{\Gamma(1+\gamma+n)},
\nonumber \\
\frac{1}{(-\gamma-n)_m (1+\gamma+k)_{n-k-m}} = & \, 
(-1)^m \, \frac{\Gamma(1+\gamma+k)}{\Gamma(1+\gamma+n)}
\nonumber 
\end{align}
and writing $(\gamma+k)\Gamma(\gamma+k) = \Gamma(1+\gamma+k)$, the term 
$\Gamma(1+\gamma+k)$ factors out and we obtain
\begin{equation}
U_{n-k}^{(n,k)} = 
\frac{\Gamma(1+\gamma+k)}{(\beta+k)\Gamma(1+\gamma+n)} \, X_{n-k}^{(n,k)}
\label{DefUbis}
\end{equation}
where the sum 
$$
X_{n-k}^{(n,k)} = \sum_{m=0}^{n-k} 
\frac{(-(\beta+n)\nu)_m}{m!}\cdot 
\frac{(\beta+n-m)((\beta+k)\nu)_{n-k-m}}{(n-k-m)!}
$$
is independent of $\gamma$. Splitting the term $(\beta + n - m)$ of the second factor inside this sum into $(\beta + n)$ and $-m$, 
$X_{n-k}^{(n,k)}$ can be written as the sum of convolution terms
\begin{align}
X_{n-k}^{(n,k)} = \, [x]^{n-k} 
\biggl \{ & \, (\beta+n) \cdot h(x;(\beta+n)\nu) \, h(x;-(\beta+k)\nu) 
\; - 
\nonumber \\
& \, x \cdot 
\left ( \frac{\partial h}{\partial x}(x;(\beta+n)\nu) \right )
\, h(x;-(\beta+k)\nu) \biggr \}
\label{DefX}
\end{align}
where we set
$h(x;\lambda) =\sum_{m \geqslant 0} (-\lambda)_m \, x^m/m! 
= (1-x)^{\lambda}$ for $\vert x \vert < 1$ and any $\lambda$; applying this definition of $h(x;\lambda)$ successively to arguments 
$\lambda = (\beta + n)\nu$ and $\lambda = -(\beta + k)\nu$ then enables us to reduce (\ref{DefX}) to
\begin{align}
X_{n-k}^{(n,k)} = 
& \, [x]^{n-k}\left\{(\beta + n)(1-x)^{(n-k)\nu} + 
(\beta+n)\nu x(1-x)^{(n-k)\nu-1}\right\}
\nonumber\\
= & \, (\beta+n) \, \frac{((k-n)\nu)_{n-k}}{(n-k)!} + 
(\beta+n)\nu \, \frac{(1+(k-n)\nu)_{n-k-1}}{(n-k-1)!}.
\label{X1}
\end{align}
After simplification, the latter expression eventually yields
$X_{n-k}^{(n,k)} = (\beta + n) \, \delta_{n,k}$ 
(for $n = k$, note that the second term in (\ref{X1}) is zero since it has denominator $(-1)! = \Gamma(0) = \infty$); the latter equality and 
(\ref{DefUbis}) together provide $U_{n-k}^{(n,k)} = \delta_{n,k}$. The inversion condition (\ref{Invers0}) is therefore fulfilled for all  
$n, \, k \geqslant 1$. We conclude that inverse relation 
(\ref{eq:inversionR}) holds for any pair of sequences 
$(S_n)_{n \geqslant 1}$ and $(T_n)_{n \geqslant 1}$.
\end{proof}


\section{Generating functions}
\label{GF}


Functional relations are now derived for the Ordinary (resp. Exponential) Generating Functions, OGFs (resp.  EGFs), of sequences related by the inversion formula.
 
\subsection{Relations for OGF's}
When $\alpha = \gamma$, the inversion formula 
(\ref{eq:inversionR}) translates into reciprocal relations for the O.G.F.'s  of related sequences $(S_n)_{n \in \mathbb{N}^*}$ and 
$(T_n)_{n \in \mathbb{N}^*}$; note that the restriction 
$\gamma \notin -\mathbb{N}$ in Theorem \ref{PropIn} cancels out when 
$\alpha = \gamma$. There is generally no such explicit relation, however,  when $\alpha \neq \gamma$.
     
\begin{corol}
\textbf{For given complex parameters $x, \, \nu, \, \beta$ and $\gamma$, let $(S_n)_{n \in \mathbb{N}^*}$ and $(T_n)_{n \in \mathbb{N}^*}$ be sequences related by the inversion formulas (\ref{eq:inversionR}) of Theorem \ref{PropIn}, that is, 
$S = \mathbf{B}(x,\nu;\gamma, \beta,\gamma) \cdot T \Leftrightarrow T = \mathbf{A}(x,\nu;\gamma, \beta,\gamma) \cdot S$.}

\textbf{Denote by $\mathfrak{G}_S(z)$ and $\mathfrak{G}_T(z)$ the formal O.G.F.'s of $S$ and $T$, respectively. Defining the mapping $\Xi$ 
(depending on parameters $x$ and $\nu$) by}
\begin{equation}
\Xi(z) = \frac{z}{z-1}\left(\frac{1-z}{1-(1-x)z}\right)^{\nu},
\label{DefXi}
\end{equation}
\textbf{the relations}
\begin{equation}
\left\{
\begin{array}{ll}
\mathfrak{G}_S(z) = \displaystyle 
\left [ \frac{1-\nu}{1-z} + \frac{\nu}{1-(1-x)z} \right ] 
\frac{(-\Xi(z)/z)^{\beta}}{(1-z)^{\gamma-\beta}} \cdot  
\mathfrak{G}_T(\Xi(z)), 
\\ \\
\mathfrak{G}_T(\xi) = \displaystyle 
\left [ \frac{1-\nu}{1-\Omega(\xi)} + 
\frac{\nu}{1-(1-x)\Omega(\xi)} \right ]^{-1} 
\frac{(-\Omega(\xi)/\xi)^{\beta}}{(1-\Omega(\xi))^{\beta-\gamma}} \cdot  \mathfrak{G}_S(\Omega(\xi)) 
\end{array} \right.
\label{eq:OGFRelation}
\end{equation}
\textbf{hold, where $\Omega$ is the inverse mapping 
$\Xi(z) = \xi \Leftrightarrow z = \Omega(\xi)$.} 
\label{corOGF}
\end{corol}

\begin{proof}
\textbf{a)} From the definition (\ref{DefABxNU}) of matrix $\mathbf{B}(x,\nu;\gamma, \beta,\gamma)$, the generating function of the sequence $S = \mathbf{B}(x,\nu;\gamma, \beta,\gamma) \cdot T$ is given by 
\begin{align}
& \mathfrak{G}_S(z) = \sum_{n \geqslant 1} z^n 
\left ( \sum_{k=1}^n B_{n,k}(x,\nu;\gamma,\beta,\gamma) T_k \right ) \; = 
\nonumber \\
& \sum_{n \geqslant 1} z^n 
\sum_{k=1}^n T_k  (-1)^k \binom{n+\gamma}{k+\gamma} 
\biggl [\frac{\gamma+k}{\beta+k} \, F(k-n,(\beta+k)\nu;\gamma+k;x) \; + 
\nonumber \\
& \qquad \qquad \qquad \qquad \qquad \qquad \; \, 
\frac{\beta-\gamma}{\beta+k} \, F(k-n,(\beta+k)\nu;1+\gamma+k;x) \biggr ]
\nonumber
\end{align}
that is,
\begin{align}
\mathfrak{G}_S(z) = & \, 
\sum_{k \geqslant 1}(-z)^k T_k 
\biggl [ \frac{\gamma+k}{\beta+k} \, 
U(\gamma+k,(\beta+k)\nu,\gamma+k;z,x) \; + 
\nonumber \\
& \qquad \qquad \qquad 
\frac{\beta-\gamma}{\beta+k} \, U(\gamma+k,(\beta+k)\nu,1+\gamma+k;z,x) \biggr ]
\label{GSForm1}
\end{align}
where we define $U(\alpha_1,\alpha_2,\alpha_3;z,x)$ by
\begin{equation}
U(\alpha_1,\alpha_2,\alpha_3;z,x) = \sum_{n \geqslant 0}\frac{(1+\alpha_1)_n}{n!} \,z^n \, F(-n,\alpha_2;\alpha_3;x).
\end{equation}
Applying definition (\ref{DefFpoly}) to Hypergeometric polynomial 
$F(-n,\alpha_2,\alpha_3;x)$ and writing $(-n)_m = (-1)^m \, n!/(n-m)!$, then interchanging the summation order and using the index change 
$n \mapsto n - k$, an expression for $U$ is easily derived in terms of another Gauss hypergeometric function, namely
\begin{equation}
U(\alpha_1,\alpha_2,\alpha_3;z,x) = 
(1-z)^{-1-\alpha_1} \cdot 
F \left( 1+\alpha_1,\alpha_2;\alpha_3;\frac{xz}{z-1} \right).
\label{DefUU}
\end{equation}
Setting $\alpha_1 = \gamma+ k$, $\alpha_2 = (\beta+k)\nu$ and either 
$\alpha_3 = \gamma + k$ or $\alpha _3 = 1 + \gamma + k$ in 
(\ref{DefUU}), the right-hand side of (\ref{GSForm1}) then reads
\begin{align}
& \frac{\gamma+k}{\beta+k}\, U(\gamma+k,(\beta+k)\nu,\gamma+k;z,x) + 
\frac{\beta-\gamma}{\beta+k} \, U(\gamma+k,(\beta+k)\nu,1+\gamma+k;z,x) 
\; = 
\nonumber \\
& (1-z)^{-1-\gamma-k} 
\biggl [ \frac{\gamma+k}{\beta+k}\, 
F\left(1+\gamma+k,(\beta+k)\nu;\gamma+k;\frac{xz}{z-1}\right) \; +
\nonumber \\
& \frac{\beta-\gamma}{\beta+k} \, 
F\left(1+\gamma+k,(\beta+k)\nu;1+\gamma+k;\frac{xz}{z-1}\right) \biggr ]
\nonumber
\end{align}
which, after invoking known identities 
$F(\alpha,\beta;\alpha;x) = (1-x)^{-\beta}$ together with 
$F(1+\alpha,\beta;\alpha;x) = (1-x)^{-\beta} + 
\beta x(1-x)^{-1-\beta}/\alpha$, further simplifies to
\begin{align}
& (1-z)^{-1-\gamma-k}
\biggl [ \biggl \{ \frac{\gamma+k}{\beta+k} \, 
\, \left(1-\frac{xz}{z-1}\right)^{-(\beta+k)\nu} + 
\nu\frac{xz}{z-1} \left(1-\frac{xz}{z-1}\right)^{-1-(\beta+k)\nu} 
\biggr \} 
\nonumber \\ 
& \qquad \qquad \qquad 
+ \; \frac{\beta-\gamma}{\beta+k} \; 
\left(1-\frac{xz}{z-1}\right)^{-(\beta+k)\nu} 
\biggr ] \; = 
\nonumber\\
& (1-z)^{-1-\gamma-k}\left(1-\frac{xz}{z-1}\right)^{-(\beta+k)\nu}
\left(1-\frac{\nu x z}{1-(1-x)z}\right).
\nonumber
\end{align}
Replacing the latter in the right-hand side of (\ref{GSForm1}), the expression of $\mathfrak{G}_S(z)$ then reduces to
\begin{align}
\mathfrak{G}_S(z) = & \, 
\frac{1}{1-z} \left(1- \frac{\nu xz}{1-(1-x)z}\right)
(1-z)^{-\gamma} \left( 1-\frac{xz}{z-1}\right )^{-\beta\nu} \; \times 
\nonumber \\
& \, \sum_{k \geqslant 1}(-\frac{z}{1-z} (1-\frac{xz}{z-1})^{-\nu})^k T_k,
\nonumber
\end{align}
that is,
\begin{equation}
\mathfrak{G}_S(z) = 
\frac{1}{1-z} \left(1- \frac{\nu xz}{1-(1-x)z}\right) 
\frac{(-\Xi(z)/z)^{\beta}}{(1-z)^{\gamma-\beta}} 
\; \mathfrak{G}_T(\Xi(z))
\label{GSForm2}
\end{equation}
with $\Xi(z)$ defined as in (\ref{DefXi}). Writing 
$$
\frac{1}{1-z} \left [1- \frac{\nu xz}{1-(1-x)z} \right ] = 
\frac{1-\nu}{1-z} + \frac{\nu}{1-z(1-x)},
$$
(\ref{GSForm2}) eventually provides the first relation 
(\ref{eq:OGFRelation}).

\textbf{b)} For any parameters $x$ and $\nu$, the function 
$z \mapsto \Xi(z)$ is analytic in a neighborhood of $z = 0$, with 
$\Xi(0) = 0$ and $\Xi'(z) \sim -z$ as $z \rightarrow 0$, hence 
$\Xi'(0) = -1 \neq 0$. By the Implicit Function Theorem, $\Xi$ has an analytic inverse $\Omega:\xi \mapsto \Omega(\xi)$ in a neighborhood of 
$\xi = 0$; the inversion of the first relation (\ref{eq:OGFRelation}) consequently provides the second relation (\ref{eq:OGFRelation}), as claimed. 
\end{proof}

\noindent 
Relations (\ref{eq:OGFRelation}) between formal generating series can also be understood as a functional identity between the analytic functions  
$z \mapsto \mathfrak{G}_S(z)$ and $z \mapsto \mathfrak{G}_T(z)$ in some neighborhood of the origin $z = 0$ in the complex plane. Now, Corollary 
\ref{corOGF} can be supplemented by making explicit the inverse mapping 
$\Omega$ involved in the 2nd relation (\ref{eq:OGFRelation}). To this end, we state some preliminary properties (in the sequel, $\log$ will denote the determination of the logarithm in the complex plane cut along the negative semi-axis $]-\infty,0]$ with $\log(1) = 0$).

\begin{lemma}
\textbf{Let $R(\nu) = \vert e^{-\psi(\nu)} \vert$ where}
$$
\psi(\nu) = \left\{
\begin{array}{ll}
(1-\nu)\log(1-\nu) + \nu\log(-\nu), \quad \; 
\nu \in \mathbb{C} \setminus \, [0,+\infty[,
\\ \\
(1-\nu)\log(1-\nu) + \nu\log(\nu), \; \quad \; \; \, 
\nu \in \mathbb{R}, \; 0 \leqslant \nu < 1, 
\\ \\
(1-\nu)\log(\nu-1) + \nu\log(\nu), \; \quad \; \; 
\nu \in \mathbb{R}, \; \nu \geqslant 1.
\end{array} \right.
$$
\textbf{The power series}
$$
\pmb{\Sigma}(w) = 
\sum_{b \geqslant 1} 
\frac{\Gamma(b(1-\nu))}{\Gamma(b)\Gamma(1-b\nu)} \cdot w^b, 
\qquad \vert w \vert < R(\nu),
$$
\textbf{is given by}
\begin{equation}
\pmb{\Sigma}(w) = \frac{\Theta(w)-1}{\nu \, \Theta(w) + 1 - \nu}
\label{U0}
\end{equation}
\textbf{where $\Theta:w \mapsto \Theta(w)$ denotes the unique analytic solution 
(depending on $\nu$) to the implicit equation}
\begin{equation}
1 - \Theta + w \cdot \Theta^{1-\nu} = 0, 
\qquad \vert w \vert < R(\nu),
\label{DefTheta}
\end{equation}
\textbf{verifying $\Theta(0) = 1$.}
\label{lemmU}
\end{lemma}

\noindent 
The proof of Lemma \ref{lemmU} is detailed in Appendix \ref{A3}. 

\begin{corol}
\textbf{For all $\nu \in \mathbb{C}$ and $x \neq 0$, the inverse mapping 
$\Omega$ of $\Xi$ defined in (\ref{DefXi}) can be expressed by}
\begin{equation}
\Omega(\xi) = 
\frac{\pmb{\Sigma}(x \, \xi)}{(1-x(1-\nu)) \, \pmb{\Sigma}(x \, \xi) - x}, 
\qquad \vert \xi \vert < \frac{R(\nu)}{\vert x \vert},
\label{InvXI}
\end{equation}
\textbf{in terms of power series $\pmb{\Sigma}(\cdot)$ defined in Lemma 
\ref{lemmU}.}
\label{corOGFbis}
\end{corol}

\begin{proof}
\textbf{\textit{(i)}} The homographic transform $h:z \mapsto \theta$ with
$\theta = (1-z)/(1-z(1-x))$ is an involution, with inverse $h^{-1}$ given by
\begin{equation}
z = h^{-1}(\theta) = \frac{1-\theta}{1-\theta(1-x)}.
\label{MobInv}
\end{equation}
Let then $\xi = \Xi(z)$ with function $\Xi$ defined as in (\ref{DefXi}); we first claim that the corresponding $\theta = h(z)$ equals 
$\theta = \Theta(x \, \xi)$ where $\Theta$ is the function defined by the implicit equation (\ref{DefTheta}). In fact, definition (\ref{DefXi}) for 
$\Xi$ and expression (\ref{MobInv}) for $z$ in terms of $\theta$ together entail
$$
\xi = \Xi(z) = \frac{z}{z-1} \, \theta^{\, \nu} = 
\displaystyle \frac{1-\theta}{1-\theta(1-x)} 
\left ( \frac{1-\theta}{1-\theta(1-x)} - 1 \right )^{-1} \, \theta^{\, \nu} = 
\frac{\theta - 1}{x \, \theta} \, \theta^{\, \nu} 
$$
and the two sides of the latter equalities give 
$1 - \theta + x \xi \theta^{1-\nu} = 0$, hence the identity 
$\theta = \Theta(x \, \xi)$, as claimed. 

\textbf{\textit{(ii)}} The corresponding inverse $z = \Omega(\xi)$ can now be expressed as follows; equality (\ref{U0}) applied to $w = x \, \xi$ can be first solved for $\Theta(x \, \xi)$, giving
$$
\Theta(x \, \xi) = \frac{1 + (1-\nu)\pmb{\Sigma}(x \xi)}
{1 - \nu \, \pmb{\Sigma}(x \xi)};
$$
it then follows from (\ref{MobInv}) and this expression of $\Theta(x \, \xi)$ that
$$
z = \Omega(\xi) = \frac{1 - \Theta(x \, \xi)}{1 - (1-x)\Theta(x \, \xi)} = 
\frac{\displaystyle 1 - \frac{1 + (1-\nu)\pmb{\Sigma}(x \xi)}
{\displaystyle 1 - \nu \, \pmb{\Sigma}(x \xi)}}
{\displaystyle 1 - (1-x)\frac{1 + (1-\nu)\pmb{\Sigma}(x \xi)}
{1 - \nu \, \pmb{\Sigma}(x \xi)}}
$$
which easily reduces to formula (\ref{InvXI}).
\end{proof}

\subsection{Relation for EGF's}
We now turn an identity between the exponential generating functions of related sequences $S$ and $T$.

\begin{corol}
\textbf{When $\alpha = 0$ and given sequences $S$ and $T$ related by the inversion formulae 
$S = B(0,\beta,\gamma;x,\nu) \cdot T \Leftrightarrow 
T = A(0,\beta,\gamma;x,\nu) \cdot S$, the EGF $\mathfrak{G}_S^*$ of the sequence $S$ can be expressed by}
\begin{align}
\mathfrak{G}_S^*(z) = \exp(z) \cdot 
\sum_{k \geqslant 1} (-1)^k T_k  \, \frac{z^k}{k!} 
\biggl[ & \, \frac{\gamma + k}{\beta + k} \, 
\Phi((\beta + k)\nu;\gamma + k;-x \, z) \; + 
\nonumber \\
& \, \frac{\beta-\gamma}{\beta+k} \Phi((\beta + k)\nu;1+\gamma + k;-x \, z)\biggr]
\label{GsExp}
\end{align}
\textbf{for all $z \in \mathbb{C}$, where $\Phi(\lambda;\mu;\cdot)$ denotes the Confluent Hypergeometric function with parameters $\lambda$, 
$\mu \notin -\mathbb{N}$.}
\label{corEGF}
\end{corol}

\begin{proof}
For $\alpha = 0$, a calculation similar to that of Corollary \ref{corOGF} gives
\begin{equation}
\mathfrak{G}_S^*(z) = \sum_{n \geqslant 0} \frac{z^n}{n!} 
\left ( \sum_{k=1}^n B_{n,k}(x,\nu;0,\beta,\gamma) T_k \right ) = I + J
\label{E0}
\end{equation}
where terms $I$ and $J$, after interchanging the summation order between 
$n$ and $k$ and using the index change $m = n-k$, can be written as 
$$
\left\{
\begin{array}{ll}
I = \displaystyle \sum_{k \geqslant 1} (-1)^k T_k \frac{z^k}{k!} \, 
\frac{\gamma + k}{\beta + k} R_k(z,x;\beta,\gamma), 
\\ \\
J = \displaystyle \sum_{k \geqslant 1} (-1)^k T_k \frac{z^k}{k!} \, 
\frac{\gamma + k}{\beta + k} R_k(z,x;\beta,1+\gamma)
\end{array} \right.
$$
respectively, with
$$
R_k(z,x;\beta,\gamma) = \sum_{m \geqslant 0} \frac{z^m}{m!} \, 
F(-m,(\beta+k)\nu;\gamma+k;x).
$$
Writing $F(-m,\eta;\zeta;x) = m! \sum_{0 \leqslant j \leqslant m} 
(\eta)_j(-x)^j/\{j!(m-j)!(\zeta)_j\}$ after definition (\ref{DefFpoly}) for any $\eta$ and $\zeta$, the latter sum $R_k(z,x;\beta,\gamma)$ reduces to
$$
R_k(z,x;\beta,\gamma) = \sum_{j \geqslant 0} 
\frac{((\beta+k)\nu)_j}{(\gamma+k)_j \, j!} \, (-x)^j 
\sum_{m \geqslant j} \frac{z^m}{(m-j)!} = e^z \, 
\sum_{j \geqslant 0} 
\frac{((\beta+k)\nu)_j}{(\gamma+k)_j \, j!} \, (-x z)^j;
$$
from the expansion of $\Phi((\beta+k)\nu;\gamma + k;-xz)$ in powers of 
$-xz$, we then obtain $R_k(z,x;\beta,\gamma) = 
e^z\Phi((\beta+k)\nu;\gamma + k;-x z)$. Applying this identity to each sum 
$I$ and $J$ above, equality (\ref{E0}) provides (\ref{GsExp}).
\end{proof}

\section{Conclusions}

As argued in the Introduction, the explicit inversion of the five parameters family of lower-triangular matrices 
$\mathbf{A}(x,\nu;\alpha,\beta,\gamma)$ has been motivated by the resolution of linear system (\ref{T0}) whose coefficients depend on a specific family of Gauss Hypergeometric polynomials. Other important applications of inversion formulas involving Gauss Hypergeometric polynomials (such as Jacobi, Chebyshev, Ultraspherical,etc.) can also be found in weighted quadrature rules \cite{Esl05} whose errors can be controlled by some generalized classical inequalities \cite{Mas09}.    

The general inversion criterion stated in Theorem 
\ref{theoMaininversionR} could be possibly applied to other examples of 
so-called \textit{independent} sequences $(a_m)$ and $(b_m)$ in order to obtain new instances of inversion formulas. Similarly, remarkable functional identities can be derived through Corollary \ref{corOGF} for OGF's. As to Corollary \ref{corEGF} for EGF's, 
a further application  of relation (\ref{GsExp}) to the specific matrix 
(\ref{T0TER}) seems promising as it can provide interesting integral representations for the associated EGF $\mathfrak{G}_E^*$ of the solution 
$E = (E_k)_{k \geqslant 1}$. This is an object of forthcoming study.




\section{Appendix}


\subsection{Proof of Lemma \ref{lemm1}}
\label{A1}
\textbf{a)} By the reflection formula 
$\Gamma(z)\Gamma(1-z) = \pi/\sin(\pi \, z)$, $z \notin -\mathbb{N}$ 
\cite[Sect.5.5.3]{NIST10} applied to the argument $z = r-\mu$, the generic term $d_r(\lambda,\mu)$ of the sum $D_N(\lambda,\mu)$ equivalently reads
$$
d_r(\lambda,\mu) = \frac{(-1)^r}{\Gamma(1+r-\lambda)\Gamma(1-r+\mu)} = 
- \frac{\sin(\pi \mu)}{\pi} \, \frac{\Gamma(r-\mu)}{\Gamma(1+r-\lambda)}
$$
and Stirling's formula \cite[Sect.5.11.3]{NIST10} entails that 
$d_r(\lambda,\mu) = O(r^{\lambda-\mu-1})$ for large $r$; the series 
$\sum_{r \geqslant 0} d_r(\lambda,\mu)$ is thus convergent if and only if 
$\mathrm{Re}(\mu) > \mathrm{Re}(\lambda)$. Write then the finite sum 
$D_N(\lambda,\mu)$ as the difference
\begin{align}
& \, \sum_{r=0}^{+\infty} \frac{(-1)^r}{\Gamma(1+r-\lambda)\Gamma(1-r+\mu)} - 
\sum_{r=N}^{+\infty}\frac{(-1)^r}{\Gamma(1+r-\lambda)\Gamma(1-r+\mu)} \; = 
\nonumber \\
& \, \sum_{r=0}^{+\infty}\frac{(-1)^r}{\Gamma(1+r-\lambda)\Gamma(1-r+\mu)} - 
\sum_{r=0}^{+\infty}\frac{(-1)^{r+N}}{\Gamma(1+r+N-\lambda)\Gamma(1-r-N+\mu)};
\nonumber
\end{align}
applying similarly the reflection formula to the argument $z = r-\mu+N$ for the second sum, we obtain
\begin{align}
D_N(\lambda,\mu) & \, = \frac{\sin(\pi \, \mu)}{\pi} 
\left [ \sum_{r=0}^{+\infty} \frac{\Gamma(r-\mu+N)}{\Gamma(1+r+N-\lambda)} 
- \sum_{r=0}^{+\infty} \frac{\Gamma(r-\mu)}{\Gamma(1+r-\lambda)} \right ]
\nonumber \\
& \, = \frac{\sin(\pi \, \mu)}{\pi} 
\left [ \sum_{r=0}^{+\infty} 
\frac{(N-\mu)_r\Gamma(r-\mu)}{(1+N-\lambda)_r\Gamma(1+N-\lambda)} 
- \sum_{r=0}^{+\infty} \frac{(-\mu)_r\Gamma(-\mu)}{(1-\lambda)_r\Gamma(1-\lambda)} \right ]
\nonumber
\end{align}
when introducing Pochhammer symbols of order $r$, hence
\begin{align}
D_N(\lambda,\mu) = \frac{\sin(\pi \, \mu)}{\pi} 
\Bigl [ & \, \frac{\Gamma(N-\mu)}{\Gamma(1+N-\lambda)} \, F(1,N-\mu;1+N-\lambda;1) \; - 
\nonumber \\
& \, \frac{\Gamma(-\mu)}{\Gamma(1-\lambda)} \, F(1,-\mu;1-\lambda;1) 
\Bigr ]
\nonumber
\end{align}
in terms of the Hypergeometric function $F$. Now, recall the identity 
\cite[Sect.9.122.1]{GRAD07}
\begin{equation}
F(\alpha,\beta;\gamma;1) = \frac{\Gamma(\gamma)\Gamma(\gamma-\alpha-\beta)}
{\Gamma(\gamma-\alpha)\Gamma(\gamma-\beta)}, \qquad 
\mathrm{Re}(\gamma) > \mathrm{Re}(\alpha + \beta);
\label{HyperF1}
\end{equation}
when applying (\ref{HyperF1}) to the values $\alpha = 1$, $\beta = N-\mu$, 
$\gamma = 1 + N -\lambda$ (resp. $\alpha = 1$, $\beta = -\mu$, 
$\gamma = 1 - \lambda$), the latter sum $D_N(\lambda,\mu)$ consequently reduces to 
\begin{equation}
D_N(\lambda,\mu) = \frac{\sin(\pi \, \mu)}{\pi} 
\frac{\Gamma(\mu-\lambda)}{\Gamma(1-\lambda+\mu)} 
\left [ \frac{\Gamma(N-\mu)}{\Gamma(N-\lambda)} - \frac{\Gamma(-\mu)}{\Gamma(-\lambda)} \right ], \quad \mathrm{Re}(\mu) > \mathrm{Re}(\lambda).
\label{A11}
\end{equation}
By the reflection formula for function $\Gamma$ again, we have
$$
\Gamma(N-\mu)\Gamma(1-N+\mu) = - \frac{(-1)^N \pi}{\sin(\pi \mu)}, \qquad 
\Gamma(-\mu)\Gamma(1+\mu) = - \frac{\pi}{\sin(\pi \mu)},
$$
so that expression (\ref{A11}) eventually yields 
\begin{align}
D_N(\lambda,\mu) & \, = - \frac{\Gamma(\mu-\lambda)}{\Gamma(1-\lambda+\mu)} 
\left [ \frac{(-1)^N}{\Gamma(N-\lambda)\Gamma(1-N+\mu)} - 
\frac{1}{\Gamma(-\lambda)\Gamma(1+\mu)} \right ] 
\nonumber \\
& \, = \frac{1}{\lambda-\mu} 
\left [ \frac{(-1)^N}{\Gamma(N-\lambda)\Gamma(1-N+\mu)} - 
\frac{1}{\Gamma(-\lambda)\Gamma(1+\mu)} \right ]
\nonumber
\end{align}
which states the first identity (\ref{Sn}) for 
$\mathrm{Re}(\mu) > \mathrm{Re}(\lambda)$.

\textbf{b)} The reflection formula for $\Gamma$ applied to $z = r-\lambda$ enables us to write 
\begin{align}
D_N(\lambda,\lambda) & \, = 
\sum_{r=0}^{N-1} \frac{(-1)^r}{\Gamma(1+r-\lambda)\Gamma(1-r+\lambda)} 
= - \frac{\sin(\pi \lambda)}{\pi} \sum_{r=0}^{N-1} 
\frac{\Gamma(r-\lambda)}{\Gamma(1-r+\lambda)} 
\nonumber \\
& \, = - \frac{\sin(\pi \lambda)}{\pi} \sum_{r=0}^{N-1} \frac{1}{r-\lambda} 
= \frac{\sin(\pi \lambda)}{\pi} 
\left [ \psi(-\lambda) - \psi(N-\lambda) \right ]
\nonumber
\end{align}
after the expansion formula \cite[Chap.5, Sect.5.7.6]{NIST10} for the function $\psi$ and the second identity (\ref{Sn}) for $\mu = \lambda$ follows.

\textbf{c)} The first identity (\ref{Sn}) stated for 
$\mathrm{Re}(\mu) > \mathrm{Re}(\lambda)$ defines an analytic function of variables $\lambda \in \mathbb{C}$ and $\mu \in \mathbb{C}$ for 
$\mu \neq \lambda$; besides, it is easily verified that this function has the limit given by $D_N(\lambda,\lambda)$ when $\mu \rightarrow \lambda$. On the other hand, the finite sum $D_N(\lambda,\mu)$ defines itself an entire function of both variables $\lambda \in \mathbb{C}$ and 
$\mu \in \mathbb{C}$; by analytic continuation, identity (\ref{Sn})  consequently holds for any pair 
$(\lambda,\mu) \in \mathbb{C} \times \mathbb{C}$ $\blacksquare$

\subsection{Proof of Lemma \ref{lemmU}}
\label{A3}
\textbf{a)} We first determine the convergence radius of the power series 
$\pmb{\Sigma}(w)$ in terms of complex parameter $\nu$. For large $b$, 

$\bullet$ if $1 - \nu \notin \; ]-\infty,0]$ and $-\nu \notin \; ]-\infty,0]$, that is, if $\nu \in \mathbb{C} \setminus [0,+\infty[$, the generic term 
$\sigma_b$ of this series is asymptotic to
$$
\sigma_b = \frac{\Gamma(b(1-\nu))}{\Gamma(b)\Gamma(1-b\nu)} 
= - \frac{1}{\nu} \cdot \frac{\Gamma(b(1-\nu))}{b! \, \Gamma(-b\nu)} 
\sim - \sqrt{\frac{-\nu}{2\pi(1-\nu)b}} \, e^{b \cdot \varphi^-(\nu)}
$$
after Stirling's formula $\Gamma(z) \sim \sqrt{2\pi} 
e^{z \log z -z}/\sqrt{z}$ for large $z$ with 
$\vert \mathrm{arg}(z) \vert \leqslant \pi - \eta$, 
$\eta > 0$ \cite[Chap.5, Sect.5.11.3]{NIST10}, and where 
$\varphi^-(\nu) = (1-\nu) \log (1-\nu) + \nu \log(-\nu)$; 

$\bullet$ if $1 - \nu \notin \; ]-\infty,0]$ and $\nu \in [0,+\infty[$ 
(the parameter $\nu$ is consequently real), that is, 
$0 \leqslant \nu < 1$, write 
$\Gamma(1-b\nu) = \pi / [\sin(\pi b \nu) \Gamma(b\nu)]$ after the reflection formula so that the generic term $\sigma_b$ is now asymptotic to
$$
\sigma_b = - \frac{1}{\nu} \cdot \frac{\Gamma(b(1-\nu))}{b! \, \pi} 
\Gamma(b\nu) \, \sin(\pi b \nu) \sim 
- \frac{1}{\nu} \sqrt{\frac{\pi}{2\nu(1-\nu)b^3}} \, \sin(\pi b \nu) 
\, e^{b \cdot \varphi(\nu)}
$$
after Stirling's formula (ibid.) and where
$\varphi(\nu) = (1-\nu) \log (1-\nu) + \nu \log(\nu)$; 

$\bullet$ finally if $\nu - 1 \in [0,+\infty]$, that is, if 
$\nu \geqslant 1$, write $\Gamma(1-b\nu) = 
\pi / [\sin(\pi b \nu) \Gamma(b\nu)]$ together with  
$\Gamma(1-b(1-\nu)) = \pi / [\sin(\pi b (1-\nu)) \Gamma(b(1-\nu))]$ after the reflection formula so that the generic term $\sigma_b$ is asymptotic to
$$
\sigma_b = \frac{(-1)^{b-1}}{\nu} \cdot 
\frac{\Gamma(b\nu)}{b! \, \Gamma(1-b(1-\nu))} \sim 
\frac{(-1)^{b-1}}{\nu} \sqrt{\frac{1}{2\pi \, \nu(\nu - 1)b^3}} \, 
e^{b \cdot \varphi^+(\nu)}
$$
after Stirling's formula and where
$\varphi^+(\nu) = (1-\nu) \log (\nu-1) + \nu \log(\nu)$. 

\textbf{b)} By the latter discussion, it therefore follows that the power series $\pmb{\Sigma}(w)$ has the finite convergence radius 
$R(\nu) = \vert e^{-\psi(\nu)} \vert$ with $\psi(\nu) = \varphi^-(\nu)$, 
$\psi(\nu) = \varphi(\nu)$ or $\varphi(\nu) = \varphi^+(\nu)$ according to the value of $\nu$, as stated in Lemma \ref{lemmU}.

Now, by the above expression of $\sigma_b$ for 
$\nu \in \mathbb{C} \setminus [0,+\infty[$, write
\begin{equation}
\sigma_b = - \frac{1}{\nu} \cdot 
\frac{\Gamma(b(1-\nu))}{b! \, \Gamma(-b\nu)}  = - \frac{1}{\nu} \cdot \binom{-1 + b(1-\nu)}{b} = - \frac{1}{\nu} \cdot 
\binom{\alpha + b\beta}{b}
\label{DefSig0}
\end{equation}
for all $b \geqslant 1$, where we set $\alpha = -1$ and $\beta = 1-\nu$. From \cite[Problem 216, p.146, p. 349]{PolSze72}, it is known that 
\begin{equation}
1 + \sum_{b \geqslant 1} \binom{\alpha + b\beta}{b} w^b = 
\frac{\Theta(w)^{\alpha+1}}{(1-\beta)\Theta(w) + \beta}
\label{PolSzeg0}
\end{equation}
for any pair $\alpha$ and $\beta$, where $\Theta = \Theta(w)$ denotes the unique solution to the implicit equation 
$1 - \Theta + w\, \Theta^\beta = 0$ 
with $\Theta(0) = 1$. By expression (\ref{DefSig0}) and relation 
(\ref{PolSzeg0}) applied to the specific values $\alpha = -1$ and 
$\beta = 1-\nu$, we can consequently assert that the series 
$\pmb{\Sigma}(w)$ equals
$$
\pmb{\Sigma}(w) = \sum_{b \geqslant 1} \sigma_b \, w^b = - \frac{1}{\nu} 
\left [ \frac{1}{\nu \, \Theta(w) + 1-\nu} - 1 \right ] = 
\frac{\Theta(w)-1}{\nu \, \Theta(w) + 1-\nu}
$$
for $\vert w \vert < R(\nu)$, as claimed. The validity of equality 
(\ref{U0}) for real $\nu \in [0,+\infty[$ follows by analytic continuation $\blacksquare$


\end{document}